\documentclass[12pt]{article}

\usepackage{amsmath}
\usepackage{amssymb}
\usepackage{amsthm}
\usepackage{graphicx}
\usepackage{amscd}
\usepackage{epic, eepic}
\usepackage{url}
\usepackage{color}

\newtheorem{theorem}{\bf Theorem}[section]
\newtheorem{lemma}[theorem]{\bf Lemma}

\newtheorem{proposition}[theorem]{\bf Proposition}
\newtheorem{problem}[theorem]{\bf Problem}

\newtheorem{conj}[theorem]{\bf Conjecture}

\newtheorem{remark}[theorem]{\bf Remark}
\newtheorem{defi}[theorem]{\bf Definition}
\newtheorem{result}[theorem]{\bf Result}
\newcommand{\cP}{\mathcal{P}}
\newcommand{\cL}{\mathcal{L}}
\newcommand{\cB}{\mathcal{B}}
\newcommand{\cD}{\mathcal{D}}
\newcommand{\cS}{\mathcal{S}}

\newcommand{\PD}{\cP_\cD}
\newcommand{\LD}{\cL_\cD}
\newcommand{\PG}{\mathrm{PG}}

\newcommand{\cut}[1]{}

 \makeatletter
 \@addtoreset{equation}{section}
 \makeatother

 \makeatletter
 \@addtoreset{table}{section}
 \makeatother

\title{Dominating sets in projective planes}

\author{{\bf Tam\'as H\'eger}\\
{\small MTA--ELTE Geometric and Algebraic Combinatorics Research Group}\\
{\small H--1117 Budapest, P\'azm\'any P.\ s\'et\'any 1/C, Hungary}\\
{\small \tt{  hetamas@cs.elte.hu}} \vspace{0.5cm} \\
  {\bf Zolt\'an L\'or\'ant Nagy}\\ 
{\small MTA--ELTE Geometric and Algebraic Combinatorics Research Group}\\
{\small H--1117 Budapest, P\'azm\'any P.\ s\'et\'any 1/C, Hungary}\\
{\small \tt{  nagyzoli@cs.elte.hu}}}

\begin{document}
\maketitle

\begin{abstract} 
We describe small dominating sets of the incidence graphs of finite projective planes by establishing a stability result which shows that dominating sets are strongly related to blocking and covering sets. Our main result states that if a dominating set in a projective plane of order $q>81$ is smaller than $2q+2\lfloor \sqrt{q}\rfloor+2$ (i.e., twice the size of a Baer subplane), then it contains either all but possibly one points of a line or all but possibly one lines through a point. Furthermore, we completely characterize dominating sets of size at most $2q+\sqrt{q}+1$. In Desarguesian planes, we could rely on strong stability results on blocking sets to show that if a dominating set is sufficiently smaller than $3q$, then it consists of the union of a blocking set and a covering set apart from a few points and lines. 

\end{abstract}

\section{Introduction}

Since the pioneering work of Berge, Cockayne, Hedetniemi, Ore, Vizing and others concerning domination in graphs (see  for instance \cite{survey1}), the study of domination-related functions became a wide branch in graph theory. Given a graph $G=(V, E)$, a subset $\cD\subseteq V$ is a \emph{dominating set} of $G$ if its closed neighborhood $\cup_{v\in \cD}N[v] = V$. The domination number is the number of vertices in a smallest dominating set of $G$.

Let $\Pi_q$ be an arbitrary finite projective plane of order $q$. The purpose of this paper is to characterize small dominating sets of the incidence graph of $\Pi_q$. Generally, we do not assume that the plane is built on a finite field $\mathbb{F}_q$; if we do, we denote the plane by $\PG(2,q)$. For an introduction to finite projective and affine planes, see \cite{Beu}.

We consider finite planes as incidence structures and hence, unless causing confusion, we identify them with their incidence graphs and mix the graph theoretic and geometric terminology. Thus we refer to the vertices of this incidence graph as points and lines of the projective plane. Throughout the paper, $\cD=\PD\cup\LD$ denotes a dominating set in (the incidence graph of) a finite projective plane $\Pi_q$ where $\PD$ is a subset of the points and $\LD$ is a subset of the lines. 
For a point $P$, $[P]$ denotes the set of lines incident with $P$. Dually, if we want to emphasize the difference between a line $\ell$ and the set of points incident with $\ell$, we might write $[\ell]$ to indicate the latter.

The idea to systematically study dominating sets of (incidence graphs of) combinatorial designs appeared in the work of Goldberg et al.\ \cite{Goldberg}. Among others, they proved the accurate lower bound on the size of minimal dominating sets of projective planes of order $q$ and formulated a conjecture on the structure of dominating sets that attain this bound. Our Theorem \ref{main} verifies this conjecture and goes far beyond by establishing a stability result.

\begin{theorem}\cite{Goldberg} \label{gold}  
The domination number of the incidence graph of an arbitrary projective plane of order $q$ is $2q$.
\end{theorem}

In case of finite projective planes, the similar concept of blocking sets has been widely investigated since the late 60's. 

\begin{defi}
A set of points is a \emph{blocking set} in $\Pi_q$ if every line of $\Pi_q$ intersects the set in at least one point. A set of lines is a \emph{covering set} in $\Pi_q$ if every point of $\Pi_q$ lies on at least one line of the set.
\end{defi}

It is clear that blocking sets and covering sets are dual concepts. The smallest examples of blocking sets in $\Pi_q$ are lines; they have $q+1$ points. Blocking sets containing a full line are called \emph{trivial}. The trivial examples for covering sets are full pencils, i.e., the set of $q+1$ lines incident with an arbitrary point. Other important examples of blocking sets are Baer subplanes, which exist only if $q$ is a square, and they have $q+\sqrt{q}+1$ points. Note that we can similarly define blocking sets for affine planes. We provide more information on projective and affine blocking sets in Section 2.

The union of a blocking set and a covering set of $\Pi_q$ clearly provides an example for a dominating set in $\Pi_q$. In fact, to dominate a line $\ell$, a dominating set $\cD$ must contain either one of the points of $\ell$ or $\ell$ itself, which means that $\PD$ dominates (blocks) all lines not in $\LD$, and $\LD$ dominates (covers) all points not in $\PD$. Hence, in some sense, a small dominating set has to be close to the union of a blocking set and a covering set. Our main Theorems \ref{main} and \ref{fo2} can be considered as formal versions of this statement.

A dominating (blocking) set is \emph{minimal} if it does not contain a smaller dominating (blocking) set. Since dominating sets are closed under addition of further vertices, we only consider minimal ones. A minimal dominating set $\cD$ might contain more than one point whose removal from $\cD$ provides a set which dominates every line and point except the removed point itself. Thus although $\cD$ was minimal, removing two points from it and adding the line incident with both points would result in a smaller dominating set. This phenomenon motivates the following concept.

\begin{defi}
We call a dominating set $\cD=\PD\cup\LD$ of $\Pi_q$ \emph{stable}, if it is minimal and there is no dominating set $\cD'\supset\cD$ such that $|\cD'|=|\cD|+1$ and $\cD'$ contains a dominating set of size less than $|\cD|$.
\end{defi}

From every minimal dominating set one may get a stable one after some steps of exchanging some incident points and lines and some deletion. Essentially we claim that small stable dominating sets are the union of a blocking set and a covering set after the deletion or addition of at most two vertices.

It is convenient to define a class of dominating sets in projective planes that contain or almost contain a trivial blocking or covering set.

\begin{defi}
We call a dominating set a \emph{primal dominating set} if it contains $q$ concurrent lines or $q$ collinear points (possibly a full pencil or a full line).
\end{defi}

\begin{theorem}\label{main} 
Let $\cD$ be a minimal dominating set in the incidence graph of an arbitrary projective plane $\Pi_q$, $q\geq 5$. If $|\cD|\leq 2q+\sqrt{q}+1$, then $\cD$ is stable. Furthermore, one of the following holds:
\begin{itemize}
\item{$\cD$ is primal, and
\begin{enumerate}
\item[(i)] $|\cD|=2q$, and  $\cD$ is the union of $q$ collinear points on a line $\ell$ and $q$ concurrent lines through a point $P$ such that $P\in\ell$, $P,\ell\notin\cD$,
\item[(ii)]  $|\cD|=2q+2$, and  $\cD$ is the union of all $q+1$  points of a line $\ell$ and $q+1$ concurrent lines through a point $P$ such that $P\not\in\ell$,
\item[(iii)]  $|\cD|=2q+\sqrt{q}+1$, and either (a) there is a Baer subplane $\Pi'$ and a point $P\in\Pi'$ such that $\LD=[P]$ and $\PD=\Pi'\setminus\{P\}$, or (b) $\cD$ is the dual of the structure described in (a),
\item[(iv)]  $3q-1> |\cD| > 2q+\sqrt{q}+1$; moreover, if $\cD$ is stable, then (a) $\LD$ is a full pencil and $\PD$ consists of all but possibly one point of a nontrivial minimal blocking set of $\Pi_q$, or (b) $\LD$ is a full pencil plus a line $\ell$ and $\PD$ consists of all but possibly one points of a minimal affine blocking set of $\Pi_q\setminus \ell$, or (c) $\cD$ is the dual of the structure described in (a) or (b),
\item[(v)] $|\cD|\geq 3q-1$.
\end{enumerate}}
\item{$\cD$ is not primal, and
\begin{enumerate}
\item[(vi)]  $|\cD|\geq 2q+2\left\lfloor \sqrt{q}\right\rfloor+2$, provided that $q>81$.
\end{enumerate}}
\end{itemize}
\end{theorem}

\begin{remark}
Note that some lower bound on $q$ is essential. For example, if one considers the point set of an oval (if $q$ is odd) or hyperoval (if $q$ is even) and the skew lines to this point set, the obtained structure is clearly a non-primal dominating set of size $q+1+q(q-1)/2$ or $q+2+q(q-1)/2$, which can be smaller than $2(q+\sqrt{q}+1)$ for small values of $q$. Also, if $P,Q\in\ell$, then $([\ell]\setminus\{P,Q\})\cup(([P]\cup[Q])\setminus\{\ell\})$ is a non-stable dominating set of size $3q-1$, which is not greater than $2q+\sqrt{q}+1$ if $q\leq 4$.
\end{remark}

It is easy to see that the structures described in Theorem \ref{main} (i)-(iv) are dominating sets. They are very close to the union of a blocking set and a covering set at least one of which is trivial. Indeed, let $\cB$ be a blocking set of $\Pi_q$, and let $P$ be an arbitrary point. Then $\cD=(\cB\setminus\{P\})\cup[P]$ is clearly a dominating set of size $|\cB|+q$ or $|\cB|+q+1$, depending on whether $P\in\cB$ or $P\notin\cB$, respectively. However, $\cD$ may not be a minimal dominating set even if $\cB$ is a minimal blocking set; for example, if a line $\ell$ through $P$ contains $q$ points of $\cB$, then $\ell$ is not essential for $\cD$, so we may exclude it from $\cD$. Regarding Theorem \ref{main}: when $\cB$ is a line, we get (i) if $P\in\cB$ and (ii) if $P\notin\cB$; if $\cB$ is a Baer subplane, we get (iii) if $P\in\cB$, and we get a dominating set of size $2q+\sqrt{q}+2$ if $P\notin\cB$. Note that if we let $\cB$ be a blocking set of the affine plane $\Pi_q\setminus\ell$ for some line $\ell$ of $\Pi_q$ and $P\notin\ell$ is an arbitrary point, then $\cD=(\cB\setminus\{P\})\cup([P]\cup\{\ell\})$ is also a dominating set.

Theorem \ref{main} (vi) is sharp when $q$ is a square: the union of a Baer subplane and a dual Baer subplane meets the bound. We conjecture that this is the only case when Theorem \ref{main} (vi) is sharp; for more details, see Section \ref{fin}.

In $\PG(2,q)$, due to its algebraic manner, strong stability results on blocking sets could have been obtained, which equip us to prove much stronger correspondences between dominating, blocking and covering sets than for general projective planes.

\begin{theorem}\label{fo2}
Let $\cD$ be a (not necessarily stable) non-primal dominating set in $\PG(2,q)$, $q=p^h$, $p$ prime.
\begin{itemize}
\item If $|\cD|\leq \frac{5}{2}q-\frac{3}{2}$, then $\PD$ is a blocking set and $\LD$ is a covering set. 
\item If $|\cD|\leq \frac{8}{3}q-2$, then either $|\PD|\leq |\LD|$ and $\PD$ is a blocking set, or $|\LD|\leq |\PD|$ and $\LD$ is a covering set. 
\item If $h=1$, then $|\cD|> \frac{8}{3}q-2$.
\item If $|\cD|\leq 3q-6\sqrt{q}+3$, $h\geq2$ and $p\geq 200$, then either $|\PD|\leq|\LD|$ and $\PD$ can be extended to a blocking set by adding at most three points to it, or $|\LD|\leq|\PD|$ and $\LD$ can be extended to a covering set by adding at most three lines to it.
\item If $|\cD|\leq 3q-6\sqrt{q}+3$ and $h=1$, then either $|\PD|\leq |\LD|$ and $\PD$ contains $q-3$ collinear points, or $|\LD|\leq |\PD|$ and $\LD$ contains $q-3$ concurrent lines. (That is, the smaller of $\PD$ and $\LD$ can be extended to a trivial blocking or covering set by adding at most four points or lines.)
\end{itemize}
\end{theorem}

The paper is built up as follows. In Section 2, we briefly introduce the results and terminology regarding blocking sets we need in the sequel; then we proceed by proving our main Theorem \ref{main} in three steps. First we describe stable primal dominating sets and verify parts (i) to (v) in Subsection \ref{subs:primal}. Then we show that a small minimal non-primal dominating set $\cD$ cannot have long secants. Next we prove that if $\cD$ has only short secants, then the desired bound $|\cD|\geq 2q+2\left\lfloor \sqrt{q}\right\rfloor+2$ holds, confirming the missing part of the main theorem. Let us note that it is quite technical to handle the case when $q$ is not a square. In Section 3, we treat the Desarguesian case and prove Theorem \ref{fo2}. We end with some open problems and general remarks on the topic in Section 4.

\section{Proof of the main theorem}

Throughout this section, $\cD=\PD\cup\LD$ is a dominating set in an arbitrary projective plane $\Pi_q$.

\subsection{Results on blocking sets}\label{subs:bl}

It is easy to see that the set of all points of an arbitrary line of $\Pi_q$ is a blocking set of size $q+1$ (these example are called \emph{trivial} blocking sets). Let us give a quick result.

\begin{proposition}[folklore]\label{bltrivi}
A set of $k$ points blocks at most $kq+1$ lines. Equality holds if and only if the $k$ points are collinear.
\end{proposition}
\begin{proof}
Let $\{P_1,\ldots,P_k\}$ be a set of $k$ points. $P_1$ blocks exactly $q+1$ lines. For $i\geq 2$, as $P_jP_i$ is already blocked by $P_j$ ($1\leq j\leq i-1$), $P_i$ blocks at most $q$ new lines with equality if and only if $P_1,\ldots,P_i$ are collinear. Hence the total number of lines blocked is $q+1+(k-1)q$. The case of equality follows immediately.
\end{proof}

The above easy calculation shows that any blocking set must contain at least $q+1$ points and in case of equality, it is the point set of a line. A blocking set is \emph{nontrivial} if it does not contain a full line. A point of a blocking set $\cB$ is called \emph{essential} (for $\cB$) if $\cB\setminus\{P\}$ is not a blocking set, and a blocking set is minimal if and only if all its points are essential.

A \emph{Baer subplane} in $\Pi_q$ is a subplane of order $\sqrt{q}+1$ (so $q$ must be a square). It is well-known that a Baer subplane $\Pi'$ has $q+\sqrt{q}+1$ points and every line intersects $\Pi'$ in either one point (i.e., the line is a tangent line) or $\sqrt{q}+1$ points (i.e., the line is a ($\sqrt{q}+1$)-secant), every point of $\Pi'$ is incident with exactly $\sqrt{q}+1$ ($\sqrt{q}+1$)-secants, and every point not in $\Pi'$ is incident with exactly one ($\sqrt{q}+1$)-secant. Hence a Baer subplane is a blocking set, and the set of ($\sqrt{q}+1$)-secants of a Baer subplane (that is, a dual Baer subplane) is a covering set, both of size $q+\sqrt{q}+1$. Baer subplanes always exist in Desarguesian projective planes of square order. Bruen gave a combinatorial characterization of the smallest nontrivial blocking sets. 

\begin{result}\cite{Bruen}\label{Bruen}
A nontrivial blocking set in an arbitrary projective plane $\Pi_q$ of order $q$ has at least $q + \sqrt{q} + 1$ points. Moreover, if this lower bound is met, then the blocking set consists of the points of a Baer subplane of $\Pi_q$.
\end{result}

For general affine planes, the best but probably not sharp lower bound known is similar.

\begin{result}\cite{Bier, BT}\label{affin}
If $\cB$ is a blocking set in an arbitrary affine plane $\mathcal{A}_q$ of order $q$, then $|\cB|\geq q+\sqrt{q}+1$. 
\end{result}

Since Bruen's paper, many results have been obtained which describe the possible size and structure of blocking sets, mostly in Desarguesian projective planes. See \cite{survey} for more details.

\subsection{Primal dominating sets}\label{subs:primal}

First we treat primal dominating sets that contain $q$ concurrent lines. Due to duality, it is enough to handle this case. 

\begin{lemma}\label{minq1}
Let $\cD=\PD\cup\LD$ be a minimal dominating set in an arbitrary projective plane $\Pi_q$ of size less than $3q-1$. Suppose that there is a point $P$ such that $|[P]\cap\LD| = q$. Let $\ell$ be the unique line in $[P]\setminus\LD$. Then $\LD=[P]\setminus\ell$ and $\PD=[\ell]\setminus\{P\}$; hence $|\cD|=2q$ and $\cD$ is stable.
\end{lemma}
\begin{proof}
We have two cases: either every point of $\ell$ except possibly $P$ belongs to $\PD$ or there exists $Q\in \ell$, $Q\neq P$ such that $Q\notin \PD$. In the former case, $([\ell]\setminus\{P\})\cup([P]\setminus\{\ell\})\subseteq\cD$ is a dominating set of size $2q$ and, by minimality, this is $\cD$. It is easy to see that this construction is stable. In the latter case, we consider three disjoint groups of elements of $\cD$. There are $q$ lines of $\LD$ through $P$; the $q$ lines on $Q$ different from $\ell$ must either be in $\LD$ or contain a point from $\PD$, which requires $q$ more elements of $\cD$; furthermore, each one of the $q-1$ points of $\ell\setminus\{P,Q\}$ is also dominated by $\cD$ so it is either in $\PD$ or on a line of $\LD$, which requires $q-1$ more elements of $\cD$. Thus $|\cD|\geq 3q-1$, a contradiction.
\end{proof}

\begin{lemma}\label{minq2}
Let $\cD=\PD\cup\LD$ be a stable dominating set in an arbitrary projective plane $\Pi_q$ of size less than $3q-1$. Suppose that there is a point $P$ such that $|[P]\cap\LD|=q+1$. Then we have the following possibilities.
\begin{enumerate}
\item{$\LD=[P]$, $P\notin\PD$ and $\PD\cup\{P\}$ is a blocking set in $\Pi_q$ for which the only non-essential point may be $P$;}
\item{$\exists!\ell\notin[P]\colon\LD=[P]\cup\ell$, $P\notin\PD$ and $\PD\cup\{P\}$ is a blocking set in the affine plane $\Pi_q\setminus\ell$ for which the only non-essential point may be $P$.}
\end{enumerate}
\end{lemma}
\begin{proof}
As $[P]\subseteq\LD$, every point is dominated by $[P]$ and, clearly, $P\notin\PD$ as $\cD$ is minimal. If there were at least two lines in $\LD\setminus[P]$, adding their intersection point to $\cD$ provided a dominating set with at least two surplus lines, in contradiction with $\cD$ being stable. 

If there is exactly one line $\ell$ in $\LD\setminus[P]$, then $\PD\cap\ell=\emptyset$ (otherwise $\cD\setminus\{\ell\}$ was also a dominating set), hence $\PD\cup\{P\}$ is a blocking set in the affine plane $\Pi_q\setminus\ell$. If $\LD\setminus[P]$ is empty, then $\PD$ must block every line except possibly those of $\LD$, thus $\PD\cup\{P\}$ forms a blocking set of $\Pi_q$. 
Since $\cD$ is minimal, any point $Q\in\PD$ is incident with a line not in $\LD$ that is tangent to $\PD$. As $[P]\subset\LD$, these lines are tangents to $\PD\cup\{P\}$, so every point of $\PD$ is essential for the respective (affine or projective) blocking set.
\end{proof}

Lemmas \ref{minq1} and \ref{minq2} (and their duals) allow us to use Results \ref{Bruen} and \ref{affin} on projective and affine blocking sets in order to prove the assertions of Theorem \ref{main} (i)-(v). 

\begin{proof}[Proof of Theorem \ref{main}, primal case] \label{proofmain}
Let $\cD$ be a minimal primal dominating set. If $|\cD|\geq 3q-1$, we have Theorem \ref{main} (v). If $|\cD|<3q-1$, then by primality and duality we may assume that there is a point $P$ such that $|[P]\cap\LD|\geq q$. 

Suppose first that $\cD$ is stable. Then, according to Lemmas \ref{minq1} and \ref{minq2}, $P\notin\PD$ and we have the following possibilities:
\begin{itemize}
\item{$|\cD|=2q$, and $\cD$ is exactly the dominating set described in Theorem \ref{main} (i).}
\item{$\LD=[P]$ and $\cB=\PD\cup\{P\}$ is a projective blocking set. 

If $\cB$ is minimal, then it is a line, a Baer subplane, or $|\cB|>q+\sqrt{q}+1$ (Result \ref{Bruen}).
If $\cB$ is a line, then $\PD$ consists of precisely $q$ collinear points; hence the dual of Lemma \ref{minq1} gives that $|\LD|=q$, a contradiction with $\LD=[P]$.
If $\cB$ is a Baer subplane, then $|\PD|=q+\sqrt{q}$ and $|\LD|=q+1$ give $|\cD|= 2q+\sqrt{q}+1$ (Theorem \ref{main} (iii)).
If $\cB$ is neither a line nor a Baer subplane, then $|\cD|>2q+\sqrt{q}+1$ (Theorem \ref{main} (iv)).

If $\cB$ is not minimal, then $\PD=\cB\setminus\{P\}$ is a minimal blocking set, hence either $\PD$ is a line or $|\PD|\geq q+\sqrt{q}+1$. If $\PD$ is a line, then $|\cD|=2(q+1)$ (Theorem \ref{main} (ii)). If $\PD$ is a Baer subplane, then $|\cD|>2q+\sqrt{q}+1$ (Theorem \ref{main} (iv)).}
\item{$\exists!\ell\notin[P]\colon\LD=[P]\cup\ell$, $\PD\cup\{P\}$ is an affine blocking set in $\Pi_q\setminus\ell$. Then $|\LD|=q+2$ and $|\PD|\geq q+\sqrt{q}$ (Result \ref{affin}), hence $|\cD| > 2q+\sqrt{q}+1$ (Theorem \ref{main} (iv)).}
\end{itemize}

Suppose now that $\cD$ is a smallest minimal non-stable primal dominating set, and suppose to the contrary that $|\cD|\leq 2q+\sqrt{q}+1$. (Note that if $q>4$, then $|\cD|<3q-1$ follows.) $|[P]\cap\LD|=q$ is not possible, otherwise $\cD$ would be stable by Lemma \ref{minq1}.
By the properties of $\cD$, we can obtain a stable dominating set $\cD'=\PD'\cup\LD'$ by adding a vertex $v$ to $\cD$ and removing the appearing surplus vertices (at least two) from it. We claim that $\cD'$ is primal. Suppose to the contrary. If $|[P]\cap\LD'|\geq q$, we are ready; otherwise there are two lines, $e_1$ and $e_2$ in $[P]\setminus\LD'$. As $\cD'$ is not primal, both lines contain at least one point not in $\PD'\cup\{P\}$ and, as $|\PD'|\leq |\PD|+1$, we find a point $R\notin(\PD\cup\PD')$ on, say, $e_1$. Then the $q$ lines of $[R]\setminus\{e_1\}$, the points of $[e_1]\setminus\{P,R\}$ are dominated by distinct vertices of $\cD'$, at most one of which may not be in $\cD$; hence, as $[P]\subset\LD$, we see that $|\cD|\geq q+q-1 -1+q+1=3q-1$, a contradiction. So $\cD'$ is indeed primal, stable, and $|\cD'|\leq 2q+\sqrt{q}$.

According to the above verified possibilities, $\cD'$ is one of the structures given in Theorem \ref{main} (i) and (ii), both of which are symmetric in the roles of points and lines. Hence, without loss of generality, we may assume that $v=Q$ is a point.
Considering the two possibilities, we see that $\PD'\subset [e]$ for a line $e$, $Q\in e$, and either $|\cD'|=2q$ and all lines of $[Q]\setminus\{e\}$ are tangents to $\PD'$, or $\cD'=[e]\cup[P]$ for some $P\notin e$ and all lines of $[Q]\setminus\{e,PQ\}$ are tangents to $\PD'$. In both cases, these lines must be dominated by pairwise distinct vertices of $\cD\setminus\cD'$, so we get $|\cD|\geq 2q-1+q=3q-1$ and $|\cD|\geq 2(q+1)-1+q-1=3q$, respectively, in contradiction with $|\cD|\leq 2q+\sqrt{q}+1$ (under $q>4$).
\end{proof}

\subsection{Non-primal case: short secants, improved bounds on $|\cD|$}

Having described primal stable dominating sets, we always suppose in the sequel that every line of $\Pi_q$ intersects $\PD$ in at most $q-1$ points and dually, every point of $\Pi_q$ is on at most $q-1$ lines of $\LD$. We will see that in this case, $\PD$ cannot have long secants, nor can $\LD$ densely cover any point. To prove this, we need some simple results that establish connections between the maximal number of points covered by a set of lines and its maximum coverage number.

\begin{lemma}[folklore]\label{maxcover}
Suppose that $\cL$ is a set of lines in $\Pi_q$ such that the maximum number of concurrent lines in $\cL$ is exactly $c$. Then the number of points covered by $\cL$ is at most \[c^2-(|\cL|+1)c+|\cL|(q+1)+1.\]
\end{lemma}
\begin{proof}
Let $n=|\cL|$ and $\cL=\{\ell_1,\ldots,\ell_n\}$ so that $\ell_1,\ldots,\ell_c$ are concurrent. Then for every $c+1\leq i\leq n$, $\ell_i$ covers at most $q+1-c$ points that are not covered by $\ell_1,\ldots,\ell_c$. Thus the total number of covered points is at most $1+cq+(n-c)(q+1-c)=c^2-(n+1)c+n(q+1)+1$.
\end{proof}

\begin{defi}
Let $c=c(\LD)$ be the maximal number such that there exists a point $P$ with $|[P]\cap\LD|=c$; dually, $k=k(\PD)$ is the maximal number such that there exists a line $\ell$ with $|\ell\cap\PD|=k$.
\end{defi}

\begin{lemma}\label{big}
Suppose that $|\LD|+2-q\leq c\leq q-1$. Then $|\LD|\geq 4q-2-|\cD|$.
\end{lemma}
\begin{proof}
As $\cD$ is a dominating set, $\LD$ covers every point not in $\PD$. Compared with the bound of Lemma \ref{maxcover}, we obtain
\[c^2-(|\LD|+1)c+|\LD|(q+1)+1\geq q^2+q+1-|\PD|.\]
By the assumption on $c$, the maximum of the left-hand side in $c$ is achieved for $c=q-1$ (and $c=|\LD|+2-q$), thus
\[2|\LD|+|\PD|\geq 4q-2.\]
Substituting $|\PD|+|\LD|=|\cD|$, we obtain the stated formula.
\end{proof}

We will also need an easy bound.

\begin{proposition}\label{minqq}
$|\cD|\geq q^2+q-(q-1)|\PD|$ and $|\cD|\geq q^2+q-(q-1)|\LD|$. In particular, if $|\cD| < 3q-1$, then $|\PD|\geq q$ and $|\LD|\geq q$.
\end{proposition}
\begin{proof}
As every line not blocked by $\PD$ must be in $\LD$, Proposition \ref{bltrivi} gives $|\cD|=|\PD|+|\LD|\geq |\PD|+q^2+q+1-(|\PD|q+1)=q^2+q-(q-1)|\PD|$. If $|\cD|<3q-1$, this immediately gives $|\PD|\geq q$. By duality, we are done.
\end{proof}

Note that Proposition \ref{minqq} relies on Proposition \ref{bltrivi} only and it immediately gives $|\cD|\geq 2q$ (Theorem \ref{gold}); moreover, the characterization of equality also follows from Proposition \ref{bltrivi}.

Now we are ready to prove that if a stable dominating set is small but not primal, then $c$ and $k$ are small. More precisely,

\begin{proposition}\label{nagyszelo}
Assume that $\cD$ is not primal. Suppose that $|\cD|+|\PD|\leq 4q-3$. Then $k\leq |\PD|-q+1$. If $|\cD|+|\LD|\leq 4q-3$, then $c\leq |\LD|-q+1$. Both assertions hold if $|\cD|\leq (5q-3)/2$.
\end{proposition}
\begin{proof}
By duality, it is enough to treat the assertions regarding $c$.
If $|\cD|\leq (5q-3)/2$, then Proposition \ref{minqq} gives $|\PD|\geq q$, so $|\LD|\leq |\cD|-q$, whence $|\cD|+|\LD|\leq 4q-3$ follows. As $\cD$ is not primal, $c\leq q-1$ holds. By $|\cD|+|\LD|\leq 4q-3$, Proposition \ref{big} gives $c\leq |\LD|-q+1$.
\end{proof}

By Proposition \ref{nagyszelo}, we see that if $\PD$ is not too large, then it has no long secants. On the other hand, $\PD$ must block almost all lines (the lines of $\LD$ may not be blocked). In such a situation the standard equations have already proved useful \cite{BBSzW,Bruen}, so we use them in the next proposition, a major step in the proof. Note that by duality, $|\PD|\leq |\LD|$ may be assumed.

\begin{proposition}\label{josagvan} 
If $\cD$ is a non-primal dominating set in $\Pi_q$, $q>81$, of size $|\cD|\leq 2(q+\sqrt{q}+1)$, then $|\PD| > q+2\sqrt{q}+1-k$ holds or $|\PD| \geq q+\left\lfloor \sqrt{q}\right\rfloor+1$.
\end{proposition}

The proof of this proposition is quite laborous, so we devote a standalone subsection to it. Note that if Proposition \ref{josagvan} holds, then we are ready to confirm Theorem \ref{main}:

\begin{proof}[Proof of Theorem \ref{main}, second part] 
Suppose that $\cD$ is non-primal, $|\cD|\leq 2(q+\sqrt{q}+1)$, and $q>81$. By duality, assume that $|\PD|\leq |\LD|$. 
Proposition \ref{josagvan} yields either $|\PD| > q+2\sqrt{q}+1-k$ or $|\PD| \geq q+\left\lfloor \sqrt{q}\right\rfloor+1$. In the former case, together with Proposition \ref{nagyszelo} we get $2|\PD| > q+2\sqrt{q}+1-k+(q+k-1)$, which also implies $|\PD|\geq q+ \left\lfloor \sqrt{q}\right\rfloor +1$. Thus $|\cD|\geq 2|\PD|\geq 2(q+\left\lfloor \sqrt{q}\right\rfloor+1)$.
\end{proof}

\subsection{Proof of Proposition \ref{josagvan}}

Throughout this section we suppose that $|\cD|\leq 2(q+\sqrt{q}+1)$, $|\PD|\leq q+2\sqrt{q}+1-k$, and we are about to show that $|\PD| \geq q+\left\lfloor \sqrt{q}\right\rfloor+1$. Note that $k\geq \left\lfloor \sqrt{q}\right\rfloor+2$ implies immediately that $|\PD| \geq q+\left\lfloor \sqrt{q}\right\rfloor+1$ (Proposition \ref{nagyszelo}), so we also assume $k\leq \left\lfloor\sqrt{q}\right\rfloor+1$. Note that if $k=2$, then $|\PD|\leq q+2$ and, as all skew lines to $\PD$ are in $\LD$, we have $|\LD|\geq q^2+q+1-|\PD|(q+1)+\binom{|\PD|}{2}\geq \binom{q}{2}>|\cD|$ if $q\geq 8$. Thus we assume $k\geq 3$.

\subsubsection{The case $k\leq \sqrt{q}$}

Take an arbitrary ordering of the lines of $\Pi_q$ and let $m_i$ be the number of the points of $\PD$ on the $i$th line. Observing that $m_i=0$ implies that the corresponding line is in $\LD$, $s:=|\{i: m_i=0\}|\leq |\LD|$ follows. Then counting in two different ways the number of pairs $(P, \ell)$ with $P \in \PD\cap\ell$, and triples $(P, Q, \ell)$ with $P, Q \in \PD\cap\ell$, we obtain what are referred to as the standard equations:

\begin{eqnarray*}
\sum_{i=1}^{q^2+q+1}m_i &=& |\PD|(q+1), \\
\sum_{i=1}^{q^2+q+1}m_i(m_i-1) &=& |\PD|(|\PD|-1).
\end{eqnarray*}

Consider the sum
\begin{equation}
w(\PD):=\sum_{i: m_i>0}(m_i-1)(m_i-k) \label{totweird}
\end{equation}

As $m_i\leq k$ for all $1\leq i\leq q^2+q+1$, we have that $w(\PD)\leq 0$. On the other hand, using the standard equations we obtain
\begin{eqnarray*}
w(\PD)&=&\sum_{i: m_i>0} m_i(m_i-1)-k\sum_{i: m_i>0} m_i+\sum_{i: m_i>0}k\\
&=& |\PD|^2-(k(q+1)+1)|\PD|+k(q^2+q+1-s).
\end{eqnarray*}

Then $s\leq |\LD|$ and $|\PD|+|\LD|=|\cD|$ imply
\begin{equation}
w(\PD)\geq |\PD|^2-(kq+1)|\PD|+k(q^2+q+1-|\cD|). \label{lbweird}
\end{equation}

Using $|\cD|\leq 2(q+\sqrt{q}+1)$, $|\PD|\leq q+2\sqrt{q}+1-k$ and $q+2\sqrt{q}+1-k\leq (kq+1)/2$ (which holds if $k\geq 3$), the right hand side of inequality \eqref{lbweird} can be estimated as 
\begin{eqnarray}
0&\geq& w(\PD)\geq |\PD|^2-(kq+1)|\PD|+k(q^2+q+1-|\cD|) \nonumber\\ &\geq& (q+2\sqrt{q}+1-k)^2- (kq+1)(q+2\sqrt{q}+1-k)+ k(q^2+q+1- 2(q+\sqrt{q}+1)) \nonumber\\ &=& 
 (q+1)(\sqrt{q}-k)^2+4(q+\sqrt{q})(\sqrt{q}-k) +2(\sqrt{q}-k). \label{csodas}
\end{eqnarray}

This is in turn a contradiction if $\sqrt{q}-k >0$. Note that $k=\sqrt{q}$ is only possible if we have equality in \eqref{csodas}, hence $|\PD|= q+2\sqrt{q}+1-k= q+\sqrt{q}+1$, which fits our assertion. Thus the only case we still have to consider is $k=\left\lfloor \sqrt{q}\right\rfloor+1$.

\subsubsection{Non-existence of $\cD$ with $k=\left\lfloor \sqrt{q}\right\rfloor+1$ and $|\PD| \leq q+\left\lfloor \sqrt{q}\right\rfloor-1$, $q\geq 27$}

Substituting $k=\left\lfloor \sqrt{q}\right\rfloor+1$ and $|\PD|\leq q+\left\lfloor \sqrt{q}\right\rfloor-1\leq q+2\sqrt{q}-k$ into \eqref{lbweird}, we get an improved version of \eqref{csodas}:
\begin{eqnarray}
0&\geq& w(\PD)\geq |\PD|^2-(kq+1)|\PD|+k(q^2+q+1-|\cD|) \nonumber\\ &\geq& (q+2\sqrt{q}-k)^2- (kq+1)(q+2\sqrt{q}-k)+ k(q^2+q+1- 2(q+\sqrt{q}+1)) \nonumber\\ 
&>& -2(q+2\sqrt{q}-k)-1 + (kq+1)+ (-3q-4\sqrt{q}-1)  \nonumber\\
&\geq& q\sqrt{q} - 4q-6\sqrt{q}-1, \label{csodas2}
\end{eqnarray}
a contradiction if $q\geq 27$.

All in all, we have to exclude the only remaining case $k=\left\lfloor \sqrt{q}\right\rfloor+1$ and $|\PD| = q+\left\lfloor \sqrt{q}\right\rfloor$. It will be handled by refining the estimate $0\geq w(\PD)$ .

\subsubsection{Non-existence of $\cD$ with $k=\left\lfloor \sqrt{q}\right\rfloor+1$ and $|\PD| = q+\left\lfloor \sqrt{q}\right\rfloor$, $q>81$ }

Substituting $k=\left\lfloor \sqrt{q}\right\rfloor+1$ into \eqref{csodas} yields 
\begin{equation}
w(\PD)\geq -3q-4\sqrt{q}-1. \label{weirdlower}
\end{equation}

We introduce the following
\begin{defi}\mbox{}
     For any line $\ell$, let the \emph{weight} of $\ell$ be defined as
\[w(\ell)=\left\{
\begin{array}{ll}
0 & \mbox{ if $\ell\cap\PD=\emptyset$,} \\
(|\ell\cap\PD|-1)(|\ell\cap\PD|-k) & \mbox{ otherwise.}
\end{array}
\right.\]
We call a line $\ell$ \emph{heavy} if $w(\ell)<0$ (these are exactly the secants which are neither tangents, nor maximal secants). 

If a line $\ell$ is fixed and $P\in\ell$, let 
$w_\ell(P)=\sum_{l\colon P\in l\neq\ell} w(l)$.
\end{defi}

This definition provides that for any line $\ell$, 
\begin{equation}
w( \PD) = w(\ell)+\sum_{P\in\ell}w_\ell(P). \label{lineweird}
\end{equation}
This enables us to study how a line, depending on how many points it has in $\PD$, effects $w(\PD)$, for which \eqref{weirdlower} is a lower bound. According to the following lemma, the length of a secant of $\PD$ can take on very few values only.

\begin{lemma} \label{egyfajta}
Suppose that $k=\left\lfloor \sqrt{q}\right\rfloor+1$ and $|\PD| = q+\left\lfloor \sqrt{q}\right\rfloor$, $q>81$. Then every line intersects $\PD$ in $0$, $1$, $k-1$ or $k$ points.
\end{lemma}

\begin{proof}
Let $\ell$ be an $m$-secant, $2\leq m \leq k-2$.
Let $P\in\ell\setminus\PD$. Suppose that $P$ is covered by a $k$-secant. Let $s(P)$ be the number of skew lines to $\PD$ on $P$. Then $|\PD|\geq m+k+(q-1-s(P))$, so $s(P)\geq m+k+q-1-|\PD|=m$. Let $t$ be the number of points in $\ell\setminus\PD$ that are covered by a $k$-secant. As all skew lines are in $\LD$, we see that $t\leq |\LD|/m$. Note that $|\LD|/m+m \leq |\LD|/2+2$ whenever $2\leq m \leq |\LD|/2$ (recall that $q\leq |\LD|\leq q+\sqrt{q}+2$), thus at least 
\[q+1-m- |\LD|/m \geq \frac{q-\sqrt{q}-4}{2}\] 
points of $\ell\setminus\PD$ are not covered by any $k$-secant.

Suppose now that $P$ is not covered by any $k$-secant. As $|\PD| > q + m $, there must be another heavy line on $P$ besides $\ell$, so $ -(k-2)\geq w_\ell(P)$. Also, $-(k-2)\geq w(\ell)=(m-1)(m-k)$ which, by \eqref{lineweird} yields
\[-\left(\frac{q-\sqrt{q}-4}{2}+1\right)(k-2) \geq w(\PD)\]
hence \eqref{weirdlower} implies that
\[-\frac{q-\sqrt{q}-2}{2}(\left\lfloor \sqrt{q}\right\rfloor-1) \geq -3q-4\sqrt{q}-1,\]
a contradiction for $q > 81$.
\end{proof}

According to Lemma \ref{egyfajta}, we call a line $\ell$ a \emph{long secant} if $|\ell\cap\PD|\geq \left\lfloor \sqrt{q}\right\rfloor$. In our final step, we resolve the case when  $|\PD|=q+\left\lfloor \sqrt{q}\right\rfloor=q+k-1$ and every long secant of $\PD$ is of size $\left\lfloor \sqrt{q}\right\rfloor$ or $\left\lfloor \sqrt{q}\right\rfloor+1$.

We call a point of type $(\alpha, \beta)$ if  $\alpha$ is the number of $k$-secants and $\beta$ is the number  of $k-1$ secants through it. Then
 $\alpha(k-1)+\beta(k-2)+1=|\PD|=q+k-1$ holds. 
 The non negative solutions of this Diophantine equation are of form $\{(\alpha_0-t(k-2), \beta_0+t(k-1)), t\in \mathbb{N}\}$. Since  $\beta\in [0,\frac{q+k-2}{k-2}]\cap \mathbb{Z}$, we have at most two types of points in $\PD$; more precisely, $\beta_0>4$ implies that every point is incident to the same number of $k$-secants and $(k-1)$-secants (if $q\geq 36$).

Let  $a$ and $b$ denote the number of points of type $(\alpha_0,\beta_0)$ and of type $(\alpha_0-(k-2),\beta_0+(k-1))$, respectively. Note that $a+b=q+k-1$. We can determine the number $N$ of $k$-secants and number $N'$ of $k-1$-secants with a simple double counting argument: 

\begin{eqnarray}
N=\frac{a\alpha_0+b(\alpha_0-k+2)}{k} \ \ \mbox{ and } \ \  N'=\frac{a\beta_0+b(\beta_0+k-1)}{k-1}.\label{szeloszam}
\end{eqnarray}

The number $M$ of exterior points is $q^2+q+1-|\PD|=q^2-k+2$.

Let $p_i$ denote the number of long secants through the $i$th exterior point. Then the (dual) standard equations give 

\begin{eqnarray}\sum_{i=1}^M p_i =N(q+1-k)+N'(q+1-k+1)\label{st1}\end{eqnarray} 
and 
\begin{eqnarray}\sum_{i=1}^M \binom{p_i}{2} = \binom{N+N'}{2}-a\binom{\alpha_0+\beta_0}{2}-b\binom{\alpha_0+\beta_0+1}{2}.\label{st2}\end{eqnarray} 

Either there exists an exterior point $P$ incident with at least three long secants, or $p_i\in \{1, 2\}$ $\forall i\in \{1, \ldots, M\}$. In the former case, let $\ell$ be a skew line to $\PD$  on $P$. Then, as each non-secant line through $P$ and each point of $\ell$ must be dominated, we obtain $|\cD|\geq 3(k-1)+(q-3)+q > 2(q+\lfloor\sqrt{q}\rfloor+1)$ if $q>25$, a contradiction.

Otherwise, consider the expression \[\sum_{i=1}^M (p_i-1)(p_i-2).\]

On the one hand, it is zero according to the assumption; on the other hand, it can be evaluated applying the standard equations to obtain

\begin{eqnarray}
0=\sum_{i=1}^M (p_i-1)(p_i-2)= 2\sum_{i=1}^M \binom{p_i}{2}-2\sum_{i=1}^M p_i+2M.\label{standard}
\end{eqnarray}

The right hand side of (\ref{standard}) can be viewed as a function \[F(q,k,\beta_0, b)=2\sum_{i=1}^M \binom{p_i}{2}-2\sum_{i=1}^M p_i+2M\] of $q$, $k$, $\beta_0$ and $b$. We are about to show that $\eqref{standard}$ has no feasible solutions if $q$ is large enough, where feasibility means that $\sqrt{q}<k\leq \sqrt{q}+1$, $0\leq \beta_0\leq \frac{q+k-2}{k-2}$, $0\leq b\leq q+k-1$, which we always assume in the sequel; moreover, all these parameters are integers.

First we show that if $\beta_0\geq 2$ and $q$ is large enough, then $F(q,k,\beta_0,b)<0$, hence $\eqref{standard}$ has no feasible solutions even among real numbers.

Suppose first that $b=0$. Then the coefficient of the leading term of $F(q,k,\beta_0,0)\cdot k^2(k-1)^2$ in $q$ is increasing in $k$ on the interval $[\sqrt{q},\sqrt{q}+1]$; for $k=\sqrt{q}+1$, it is $(3-2\beta_0)$ which, if $q$ is large enough, indeed yields $F(q,k,\beta_0,0)<0$ as $\beta_0\geq 2$. 
To keep the bound on $q$ sufficiently low, applying formal and numerical calculations it can be shown that $F(q,k,\beta_0,0)$ is convex in $k$ on the interval $[\sqrt{q},\sqrt{q}+1]$ for any fixed $\beta_0\geq 0$ (if $q\geq 60$), and it is decreasing in $\beta_0$ on the interval $[0,\frac{q+\sqrt{q}-2}{\sqrt{q}-2}]$ for any fixed $k\in[\sqrt{q},\sqrt{q}+1]$ (if $q\geq 126$). As for any $q\geq 21$, $F(q,\sqrt{q},2,0)$ and $F(q,\sqrt{q}+1,2,0)$ are both negative, we conclude that if $\beta_0\geq 2$ and $q\geq 126$, then $F(q,k,\beta_0,0)<0$, hence there is no feasible solution for \eqref{standard} with $b=0$.

Now suppose $b>0$, which implies $\beta_0\leq 4$. One can calculate that in this case, if $q\geq 130$, $F(q,k,\beta_0,b)\leq F(q,k,\beta_0,0)$, where the latter one is negative as we have seen before, hence \eqref{standard} has no feasible solutions with $b\geq 0$ if $\beta_0\geq 2$ and $q\geq 130$.

Now it remains to handle the the cases (I) $\beta_0=0$ and (II) $\beta_0=1$.
Taking into consideration the integer values of (\ref{szeloszam}) and the equation $\alpha_0(k-1)=q$ (case (I)) or $\alpha_0(k-1)=q+k-2$ (case (II)), these conclude to the following possibilities:

\begin{itemize}
\item[(I.a)] $\alpha_0=k-1, \ \beta_0=1, \ k=\sqrt{q}+1$,\ \ \ moreover $k\mid 2b$ and $N'=b+k$
\item[(I.b)] $\alpha_0=k, \ \beta_0=1, \  k(k-1)=q$, \ \ \ \ \ moreover $ k \mid 2b$  and $N'=b+k+1$
\item[(I.c)] $\alpha_0=k+1, \ \beta_0=1, \ k^2-1=q$, \ \ \ moreover $ k \mid 2b-2$  and $N'=b+k+2$
\item[(II.a)] $\alpha_0=k, \ \beta_0=0, \ k^2-2k+2=q$, \ \ moreover $ k \mid 2b$   and $N'=b$
\item[(II.b)] $\alpha_0=k+1, \ \beta_0=0, \ q=k^2-k+1$, moreover $ k \mid 2b-2$  and $N'=b$.
\end{itemize}

Solving (\ref{standard}) with these parameters, only (I.a) and (II.a) give feasible integer solutions for $b\in [0, q+k-1]$, namely $b=k/2$ and $b=k$, respectively (if $q\geq 21$). These cases can be excluded by combinatorial reasoning: $b\geq 2$ implies that there exist $2$ points covering at least $2(k-1)-1$ $(k-1)$-secants together, which is larger than $N'= \frac{3}{2}k$ (case (I.a)) or $N'=k$ (case (II.a)) (if $q\geq 25$).

An exhaustive search for integer feasible solutions of \eqref{standard} in case of $q\leq 130$ shows that if $q\geq 30$, then every such solution comes with $\beta_0=0$ or $\beta_1=1$, so they can be excluded by the above argument. All in all, we obtain that if $q>81$, then \eqref{standard} has no integer feasible solutions that can be realized with a dominating set.

\section{Dominating sets of $\PG(2,q)$}

In this section, using strong results on blocking sets of Desarguesian projective planes, we strongly improve the general Theorem \ref{main}. Roughly we show that if a (not necessarily primal) dominating set of $\PG(2,q)$ is significantly smaller than $3q$, then it must (almost) contain a blocking or a covering set. Let us first collect the results that have a key role in this section.

\begin{result}[Jamison \cite{J}, Brouwer--Schrijver \cite{BS}]\label{Jamison}
A blocking set in a Desarguesian affine plane of order $q$ has at least $2q-1$ points.
\end{result}

The above result is easily seen to be equivalent with the following.

\begin{result}\label{tangents}
Let $\cB$ be a blocking set of $\PG(2,q)$, and let $P$ be an essential point for $\cB$. Then there are at least $2q+1-|\cB|$ tangents to $\cB$ through $P$.
\end{result}

\begin{result}[Blokhuis \cite{Blokhuis}]\label{Blokhuis}
A blocking set in $\PG(2,q)$, $q$ prime, has at least $\frac32(q + 1)$ points, or contains a line.
\end{result}

\begin{result}[Blokhuis--Storme--Sz\H{o}nyi \cite{BSSz}]\label{BSSz}
Let $\cB$ be a blocking set in $\PG(2,q)$, $q=p^h>16$, $p$ prime, that contains neither a line nor a Baer subplane. Let $c_2=c_3=2^{-1/3}$ and $c_p=1$ for $p>3$. Then $|\cB|\geq q+c_pq^{2/3}+1$.
\end{result}

Note that the main result of \cite{BSSz} is more general as it concerns multiple blocking sets, but we do not need it here. 

\begin{result}[Sz\H{o}nyi--Weiner \cite{SzW-sb}]\label{stbl2}
Let $\cB$ be a set of points in $\PG(2,q)$, $q=p^h$, $h\geq 2$. Denote the number of skew lines to $\cB$ by $\delta$, and suppose that $\delta\leq \frac{1}{100}pq$. Assume that $|\cB|<\frac32(q+1-\sqrt{2\delta})$. Then $\cB$ can be extended to a blocking set by adding at most
\[\frac{\delta}{2q+1-|\cB|} + \frac{1}{100}\]
points to it.
\end{result}

\begin{result}[Sz\H{o}nyi--Weiner \cite{SzW-sbp}]\label{stblp}
Let $\cB$ be a set of points of $\PG(2,q)$, $q=p$ prime, with at most $\frac32(q+1)-\varepsilon$ points. Suppose that the number $\delta$ of skew lines to $\cB$ is less than $\left(\frac23(\varepsilon+1)\right)^2/2$. Then there is a line that contains at least $q-\frac{2\delta}{q+1}$ points of $\cB$.
\end{result}

A \emph{partial cover} of $\PG(2,q)$ with $h>0$ \emph{holes} is a set of lines in $\PG(2,q)$ such that the union of these lines covers all but exactly $h$ points. We will also use the dual of the following result due to Blokhuis, Brouwer and Sz\H{o}nyi \cite{covering}.

\begin{result}[\cite{covering}]\label{holes}
A partial cover of $\PG(2, q)$ with $h > 0$ holes, not all on one line if $h>2$, has size at least $2q-1-h/2$.
\end{result}

Now we proceed with examining small dominating sets of $\PG(2,q)$. If $\cD$ is a primal dominating set, then Theorem \ref{main} can be refined in the following way.

\begin{theorem}\label{primaldes}
Let $\cD$ be a minimal primal dominating set of $\PG(2,q)$, $q>16$. Define $c_p$ as in Result \ref{BSSz}. Then we have one of the cases (i)--(v) of Theorem \ref{main} with the following additions:
\begin{itemize}
\item{case (b) of (iv) is not possible (nor its dual);}
\item{if $|\cD|<3q-1$, then $\cD$ is stable;}
\item{if $3q-1>|\cD|>2q+\sqrt{q}+1$, then either (a) $|\cD| = 2q+\sqrt{q}+2$, and equality holds if and only if there is a Baer subplane $\Pi'$ and a point $P\notin\Pi'$ such that $\LD=[P]$ and $\PD=\Pi'$ (or $\cD$ is the dual structure), or (b) $|\cD|\geq 2q+c_pq^{2/3}+1$.}
\end{itemize}
\end{theorem}
\begin{proof}
Throughout we suppose that $|\cD|<3q-1$. Suppose first that $|\cD|>2q+\sqrt{q}+1$. Then we have case (iv) of Theorem \ref{main}. Assume that $\cD$ is stable. Let $\cB$ be the minimal blocking set of cases (a) or (b). In case (a), either $\cB$ is a Baer subplane, or $|\cB|\geq q+c_pq^{2/3}+1$. In the former case, $|\cD|>2q+\sqrt{q}+1$ yields $P\notin\cB$, whence $|\cD|=2q+\sqrt{q}+2$ and the assertion follows. In the latter case, $|\cD|\geq q+1 + |\cB|-1 > 2q+c_pq^{2/3}+1$. In case (b), by Result \ref{affin}, $|\cB|\geq 2q-1$, so $|\cD|\geq q+2+|\cB|-1=3q$.

Now we claim that there cannot be a non-stable primal dominating set of size less than $3q-1$. To prove this, we follow the steps and notation of the proof of the respective part of Theorem \ref{main} (page \pageref{proofmain}). By duality we may assume that $v=Q$ is a point. Then $|\cD|\geq |\PD'|-1+|\LD'|+|t(Q)\setminus\LD'|$, where $t(Q)$ is the set of tangents to $\PD'$ through $Q$. By Result \ref{tangents}, $|t(Q)|\geq 2q+1-|\PD'|$. The details are left for the reader.
\end{proof}

For $q\leq 16$, some weaker results prior to \ref{BSSz}, see \cite[Section 5]{BBl2bl}, provide useful bounds on blocking sets that could be applied in the proof. Not to get lost in details, we omitted including these. It follows, however, that already for $q\geq9$, $|\cD|=2q+\sqrt{q}+2$ is possible if and only if $\cD$ is the union of a Baer subplane and a full pencil (or the dual structure).

Next we treat non-primal dominating sets. The next proposition can be regarded as a strengthening of Proposition \ref{nagyszelo}. Recall that $c$ is the maximal number of concurrent lines in $\LD$.

\begin{proposition}\label{pg2q}
Let $\cD=(\PD,\LD)$ be a non-primal dominating set in $\PG(2,q)$, $q=p^h$, $p$ prime. If $|\cD|+|\PD|\leq 4q-3$, then $\PD$ is a blocking set. Dually, if $|\cD|+|\LD|\leq 4q-3$, then $\LD$ is a covering set.
\end{proposition}
\begin{proof}
Suppose that $|\cD|+|\PD|\leq 4q-3$. If $\PD$ blocks all but exactly $0<h\leq|\LD|$ lines, then the dual of Result \ref{holes} yields two options. The first option is $|\PD|\geq 2q-1-|\LD|/2$, which gives $|\cD|+|\PD|\geq 4q-2$, a contradiction. The second option is that all the $h$ lines not blocked by $\PD$ go through one point, say, $P$. But then $\PD\cup\{P\}$ is a blocking set in $\PG(2,q)$, for which $P$ is an essential point. Then there are at least $2q-|\PD|$ (Result \ref{tangents}) skew lines to $\PD$ through $P$. This means $c\geq 2q-|\PD|$. On the other hand, Proposition \ref{nagyszelo} gives $c\leq |\PD|-q+1$ (under $c<q$, so if $\cD$ is not a primal dominating set). These give $2|\PD|\geq 3q-1$. By Proposition \ref{minqq}, $4q-3 \geq |\cD|\geq q^2+q-(q-1)|\LD|$, whence $|\LD|\geq q-1$ follows, so $|\cD|+|\PD|\geq 4q-2$, a contradiction. Thus we conclude that $\PD$ is a blocking set.
\end{proof}

\begin{theorem}\label{thmpg2q}
Let $\cD$ be a non-primal dominating set in $\PG(2,q)$, $q=p^h$, $p$ prime.
\begin{itemize}
\item If $|\cD|\leq (5q-3)/2$, then $\PD$ is a blocking set and $\LD$ is a covering set. 
\item If $|\cD|\leq (8q-6)/3$, then either $|\PD|\leq |\LD|$ and $\PD$ is a blocking set, or $|\LD|\leq |\PD|$ and $\LD$ is a covering set. 
\item If $|\cD|\leq 3(q+1-2\sqrt{q})$, $p\geq 200$, and $h\geq2$, then either $|\PD|\leq|\LD|$ and $\PD$ can be extended to a blocking set by adding at most three points to it, or $|\LD|\leq|\PD|$ and $\LD$ can be extended to a covering set by adding at most three lines to it.
\end{itemize}
\end{theorem}
\begin{proof}
Assume that $|\cD|\leq (5q-3)/2$. Then, as $|\LD|\geq q$ (Proposition \ref{minqq}), $|\PD|\leq |\cD|-q$ follows, hence $|\cD|+|\PD|\leq 4q-3$, thus Proposition \ref{pg2q} yields that $\PD$ is a blocking set. By duality, $\LD$ is a covering set.

Suppose that $|\cD|\leq (8q-6)/3$. By duality we may assume $|\PD|\leq|\LD|$. Then $|\PD|\leq |\cD|/2$, so $|\cD|+|\PD|\leq 3|\cD|/2\leq 4q-3$, thus Proposition \ref{pg2q} applies.

Suppose now that $|\cD|\leq 3(q+1-2\sqrt{q})$, $p\geq 200$, $h\geq 2$. As $|\PD|\geq q$, $|\LD|=|\cD|-|\PD|\leq 2q$. By duality we may assume $|\PD|\leq|\LD|$, so $|\PD|\leq \frac32(q+1-2\sqrt{q})\leq \frac32(q+1-\sqrt{2|\LD|})$. As $\frac1{100}pq\geq 2q\geq |\LD|$, Result \ref{stbl2} yields that $\PD$ can be extended to a blocking set by adding at most 
\[\frac{|\cD|-|\PD|}{2q+1-|\PD|} + \frac{1}{100}\leq \frac{\frac32(q+1-2\sqrt{q})}{2q+1-\frac32(q+1-2\sqrt{q})}+\frac1{100}=\frac{3q+3-6\sqrt{q}}{q-1+6\sqrt{q}}+\frac1{100}<4\]
points to it.
\end{proof}

\begin{theorem}\label{thmpg2p}
Let $\cD=(\PD\cup\LD)$ be a non-primal dominating set in $\PG(2,q)$, $q=p$ prime. Then 
\begin{itemize}
\item $|\cD|\geq (8q-5)/3$; 
\item if $|\cD|\leq 3q+3-6\sqrt{q}$, then either $|\PD|\leq |\LD|$ and $\PD$ contains $q-3$ collinear points, or $|\LD|\leq |\PD|$ and $\LD$ contains $q-3$ concurrent lines.
\end{itemize}
\end{theorem}
\begin{proof}
If $|\cD|\leq (8q-6)/3$, then Theorem \ref{thmpg2q} yields that, say, $|\PD|\leq|\LD|$, and $\PD$ is a blocking set. As $\cD$ is non-primal, $\PD$ must not contain a line, hence Blokhuis' celebrated Result \ref{Blokhuis} gives $|\PD|\geq\frac32(q+1)$, hence $|\cD|\geq 3(q+1)$, a contradiction.

Assume now $|\cD|\leq 3q+3-6\sqrt{q}$ and, by duality, $|\PD|\leq |\LD|$. Then $|\PD|\leq \frac32(q+1)-3\sqrt{q}$, and $|\LD|\leq 2q\leq \frac29(3\sqrt{q}+1)^2$, so Result \ref{stblp} applies with $\varepsilon=3\sqrt{q}$. Hence $\PD$ contains at least $q-4q/(q+1)>q-4$ collinear points.
\end{proof}

\begin{remark}
Some restrictions on $|\cD|$ are indeed needed in Theorem \ref{thmpg2q}. Let $\ell_0=\{P_0,P_1,\ldots,P_q\}$ be a line, let $[P_0]=\{\ell_0,\ell_1,\ldots,\ell_q\}$, and let $\ell_1=\{P_0,Q_1,Q_2,\ldots,Q_q\}$. Choose $1\leq t\leq q-1$. Let us define $\PD$ as $\{P_2,\ldots,P_q,Q_1,\ldots,Q_t\}$, and define $\LD$ as $\{\ell_2,\ldots,\ell_q,P_2Q_{t+1},\ldots,P_2Q_q\}$. Then $\PD\cup\LD$ is a dominating set of size $3q-2$, and neither $\PD$ is a blocking set, nor $\LD$ is a covering set.
\end{remark}

\begin{remark}
Theorem \ref{thmpg2q} does not remain valid in non-Desarguesian projective planes. E.g., there is a projective plane $\Pi_q$ (in fact, a Hall plane) of square order $q$ that contains Baer subplanes, and there is a suitable line $\ell\in\Pi_q$ such that there is an affine blocking set $\cB$ in $\Pi_q\setminus\ell$ of size $\lfloor 4q/3+5\sqrt{q}/3\rfloor$ \cite{DHSzV}. We may choose a dual Baer subplane $\cS$ of $\Pi_q$ that contains the line $\ell$. Then $\cB\cup\cS$ is a dominating set of $\Pi_q$ of size $\approx 7q/3<\frac52q$, yet its point set is not a blocking set of $\Pi_q$.
\end{remark}

\section{Final remarks and open problems}\label{fin}

We conjecture that $|\cD|= 2(q+ \left\lfloor \sqrt{q}\right\rfloor+1)$ in Theorem \ref{main} (vi) is only possible when $q$ is a square, $\PD$ forms a Baer subplane and $\LD$ forms a dual Baer subplane (if $q$ is large). More generally, it seems that if a small enough dominating set $\cD$ contains neither $q$ collinear points nor $q$ concurrent lines (i.e., it is not primal), then it must be the union of a blocking set and a covering set. However, the plain description of the case of equality in Theorem \ref{main} (iii) seems hard. It is closely related to our following conjecture.

\begin{conj}\label{conj}
Let $\Pi_q$ be an arbitrary projective plane of order $q$ ($q$ square, large enough), and suppose that $\Pi'$ is a Baer subplane of $\Pi_q$. If a line set $\cL$ of $\Pi_q$ of size $|\cL|\leq q+\sqrt{q}+1$ covers all points of $\Pi_q$ not in $\Pi'$, then $\cL$ is a covering set (so $\cL$ also covers the points of $\Pi'$).
\end{conj}

In the above setting, if $\cL$ is not a full pencil, then $\Pi'\cup\cL$ is a non-primal dominating set of size at most $2q+2\sqrt{q}+2$, hence it reaches equality in Theorem \ref{main} (vi), so the conjectured characterization would imply Conjecture \ref{conj}. Note that if the plane is Desarguesian, then Conjecture \ref{conj} is already implied by Theorem \ref{fo2}; moreover, with the help of stability results on blocking sets, one can prove a much stronger assertion.

From a constructive point of view, if one considers a set $\mathcal{P}$ of points and the set $\mathcal{L}$ of lines skew to $\mathcal{P}$, then adds every point to $\mathcal{P}$ which is not covered by $\mathcal{L}$, the obtained structure will be a minimal dominating set. In this way all dominating sets can be constructed whose points and lines are pairwise not incident. In another terminology, these are the $1$-dominating independent sets of the incidence graph, see \cite{GH, Nagy}. Note that avoiding incidences is not necessary in general.

We have seen a strong connection between minimal dominating sets and minimal blocking sets in $\Pi_q$. The size of minimal blocking sets in $\Pi_q$ have a well known upper bound, namely it exceed the size $q^{3/2}+1$ of a
unital \cite{BT}. Hence the question naturally rises whether similar upper bounds exist concerning minimal dominating sets. Dominating independent sets provide a wide range of examples with size close to $q^2$.

The presented stability result, Theorem \ref{main} can be viewed as a special case of the investigation of domination properties in combinatorial designs. This question was recently raised by Goldberg et al.\ in \cite{Goldberg} who proved several bounds on the domination number of designs. As dominating sets and blocking sets are somewhat similar and related concepts, let us mention that the study of blocking sets in combinatorial designs started much earlier, see e.g.\ the paper of Bruen and Thas \cite{BT}. 

In projective planes, $t$-fold blocking sets play an important role as well. On the other hand, multiple domination also has a broad literature (see e.g. \cite{fink, survey1}) which yields the following problem. 

\begin{problem} 
Describe small minimal $t$-dominating sets in arbitrary or in Desarguesian projective planes.
\end{problem}

Once small minimal dominating sets are classified, a natural extremal question would be to study the following 

\begin{problem} 
Determine the order of magnitude of the number of dominating sets, independent dominating sets, minimal dominating sets in projective planes.
\end{problem}


\begin{thebibliography}{00}

\small
\parskip -2mm

\bibitem{BBl2bl} \textsc{S.~Ball, A.~Blokhuis}, On the size of a double blocking set in PG(2,q). \textit{Finite Fields Appl.}, \textbf{2} (1996) 125--137.

\bibitem{Beu} \textsc{A.~Beutelspacher, U.~Rosenbaum}, Projective geometry: from foundations to applications. (English summary) Cambridge University Press, Cambridge, 1998.

\bibitem{Bier} \textsc{J.~Bierbrauer}, On minimal blocking sets. \textit{Archiv der Mathematik} \textbf{35}(1) (1980) 394--400.

\bibitem{Blokhuis} \textsc{A.~Blokhuis}, On the size of a blocking set in PG$(2,p)$. \textit{Combinatorica} \textbf{14}(1) (1994) 111--114.

\bibitem{survey} \textsc{A.~Blokhuis}, Blocking sets in Desarguesian planes. \textit{Combinatorics, Paul Erd\H{o}s is Eighty}, Vol. 2 (Eds. D.~Mikl\'os, V.~T.~S\'os, T.~Sz\H{o}nyi) (1996) 133--155.

\bibitem{BBSzW} \textsc{A.~Blokhuis, A.~E.~Brouwer, T.~Sz\H{o}nyi, Zs.~Weiner}, On $q$-analogues and stability theorems. \textit{Journal of Geometry}, \textbf{101}(1-2), (2011) 31--50.

\bibitem{covering} \textsc{A.~Blokhuis, A.~E.~Brouwer, T.~Sz\H{o}nyi}, Covering all points except one, \textit{J.~Algebraic Combin.} \textbf{32} (2010) 59--66. 

\bibitem{BSSz} \textsc{A.~Blokhuis, L.~Storme, T.~Sz\H onyi}, Lacunary polynomials, multiple blocking sets and Baer subplanes. \textit{J. London Math. Soc.}  \textbf{60}(2) (1999) 321--332.

\bibitem{BS} \textsc{A.~E.~Brouwer, A.~Schrijver}, The blocking number of an affine space. \textit{Journal of Combinatorial Theory, Series A} \textbf{24}(2) (1978) 251--253.

\bibitem{Bruen} \textsc{A.~A.~Bruen}, Baer subplanes and blocking sets. \textit{Bull.~Amer.~Math.~Soc.} \textbf{76}(2) (1970) 342--344.

\bibitem{BT} \textsc{A.~A.~Bruen, J.~A.~Thas}, Blocking sets. \textit{Geometriae Dedicata} \textbf{6}(2) (1977) 193--203.

\bibitem{DHSzV} \textsc{J.~De~Beule, T.~H\'eger, T.~Sz\H{o}nyi, G.~Van de Voorde}, Blocking and double blocking sets in finite planes. \textit{Submitted.}

\bibitem{FSSzW} \textsc{S.~Ferret, L.~Storme, P.~Sziklai, Zs.~Weiner}, A $t$ (mod $p$) result on weighted multiple $(n-k)$-blocking sets in PG$(n,q)$. \textit{Innov. Incidence Geom.} \textbf{6-7} (2009) 169--188.

\bibitem{fink} \textsc{J.~F.~Fink, J.~S.~Michael}, $n$-Domination in graphs. \textit{Graph theory with applications to algorithms and computer science.} John Wiley \& Sons, Inc., pp.\ 283--300, 1985.

\bibitem{GH} \textsc{W.~Goddard, M.~A.~Henning}, Independent domination in graphs: a survey and recent results. \textit{Discrete Mathematics} \textbf{313}(7) (2013) 839--854.

\bibitem{Goldberg} \textsc{F.~Goldberg, R.~Deepak, M.~Rogers}, Domination in designs. \textit{arXiv preprint} arXiv:1405.3436 (2014).

\bibitem{survey1} \textsc{T.~W.~Haynes, S.~Hedetniemi, P.~Slater}, Fundamentals of domination in graphs. CRC Press, 1998.

\bibitem{J} \textsc{R.~E.~Jamison}, Covering finite fields with cosets of subspaces. \textit{Journal of Combinatorial Theory, Series A} \textbf{22}(3) (1977) 253--266.

\bibitem{Nagy} \textsc{Z.~L.~Nagy}, On the number of $k$-dominating independent sets. \textit{arXiv preprint} arXiv:1504.03224 (2015).

\bibitem{Szonyi} \textsc{T.~Sz\H{o}nyi}, Blocking sets in Desarguesian affine and projective planes. \textit{Finite Fields and Their Applications} \textbf{3}(3) (1997) 187--202.

\bibitem{SzW-sb} \textsc{T.~Sz\H{o}nyi, Zs.~Weiner}, On the stability of small blocking sets. \textit{J.\ Algebr.\ Combin.} \textbf{40} (2014) 279--292.

\bibitem{SzW-sbp} \textsc{T.~Sz\H{o}nyi, Zs.~Weiner}, A stability theorem for lines in Galois planes of prime order. \textit{Des.\ Codes Cryptogr.} \textbf{62} (2012) 103--108.

\end{thebibliography}
\end{document}